\documentclass[reqno]{amsart}
\usepackage{amsthm,amsmath,amssymb}
\usepackage[pdftex]{graphicx}
\usepackage{braket}
\usepackage{color}

\title{On eigenvalues of double branched covers}
\author{Kouki Sato}
\date{}

\newtheorem{thm}{Theorem}
\newtheorem{prop}{Proposition}
\newtheorem{lem}[thm]{Lemma}
\newtheorem{cor}{Corollary}

\theoremstyle{definition}

\newtheorem*{remark}{Remark}
\newtheorem*{acknowledge}{Acknowledgements}

\DeclareMathOperator{\punc}{punc}

\DeclareMathOperator{\alt}{alt}
\DeclareMathOperator{\dalt}{dalt}

\begin{document}
\maketitle

\begin{abstract}
For a given knot, we study the minimal number of positive eigenvalues of the double branched cover over spanning surfaces for the knot. 
The value gives a lower bound for various genera, the dealternating number and the alternation number of knots, and we prove that Batson's bound for the non-orientable 4-genus gives an estimate of the value.
In addition, we use the value to give a necessary condition for being quasi-alternating.
\end{abstract}

\section{Introduction}
Throughout this paper, all manifolds are assumed to be smooth and compact unless otherwise stated.

For a surface $F$ properly embedded in $B^4$, let $M_F$ denote 
the double branched cover of $B^4$ over $F$,
and $b_2^+(M_F)$ ($b_2^-(M_F)$) the number of positive (resp.\ negative)
eigenvalues of the intersection form of $M_F$.
Then we can define the knot invariants
$$
b^{\pm}(K) :=
\min \left\{b^{\pm}_2(M_F) \mid F \text{ is a surface in $S^3$ with $\partial F = K$} \right\}
$$
and the knot concordance invariants
$$
b^{\pm}_*(K) :=
\min \left\{b^{\pm}_2(M_F) \mid F \text{ is a surface in $B^4$ with $\partial F = K$} \right\}.
$$
Here, we obtain $M_F$ for a surface $F$ in $S^3$ by pushing the interior
of $F$ into the interior of $B^4$ and taking the double branched cover of the resulting surface.
Recently, Greene \cite{greene} gives the following characterization of alternating knots.
\begin{thm}[Greene, \text{\cite[Theorem 1.1]{greene}}] 
\label{thm greene}
A knot $K$ is alternating if and only if $b^+(K)=b^-(K)=0$.
\end{thm}
This theorem implies that alternating knots can be thought of as the trivial knot 
in terms of $b^{\pm}(K)$,
while knots concordant to an alternating knot are like slice knots in terms of $b^{\pm}_*(K)$.
In this paper, we study the invariants $b^{\pm}$ and $b^{\pm}_*$ and their relationship to various genera, 
the dealternating number \cite{adams}, the alternation number \cite{kawauchi} and 
quasi-alternating knots \cite{oz-sz}.

Here we mention two results of this work.
The first result is an observation of Batson's bound of 
the non-orientable 4-genus \cite{batson}.
Here {\it the non-orientable 4-genus}
$\gamma_4(K)$ of a knot $K$ is the minimal first Betti number of non-orientable surfaces in $B^4$ with boundary $K$. Using the Heegaard Floer correction term of the $(-1)$-surgery along $K$ ( denoted $d(S^3_{-1}(K))$) and the knot signature $\sigma(K)$, Batson gives the inequality
$$
\gamma_4(K) \geq  \frac{\sigma(K)}{2} - d(S^3_{-1}(K))
$$
and prove that $\gamma_4$ can be arbitrarily large.
Batson's bound is strong enough to prove
$\gamma_4(T_{2n, 2n-1}) = n-1$ for any integer $n>1$ 
where $T_{2n, 2n-1}$ denotes the $(2n, 2n-1)$-torus knot, while the bound becomes a trivial inequality for any alternating knot. We found the reason of this gap; Batson's bound is essentially a lower bound for $b^+_*(K)$.
\begin{thm}
\label{thm2}
For any knot $K$, we have 
$$
\gamma_4(K) \geq b^+_*(K) \geq \frac{\sigma(K)}{2} -d(S^3_{-1}(K)).
$$
\end{thm}
Theorem \ref{thm2} implies that $b^+_*(K)$ can be arbitrarily large.
\begin{cor}
\label{cor1}
$b^+_*(T_{2n,2n-1})= n-1$ for any $n \in \mathbb{Z}_{>0}$.
In particular, $b^+_*$ can be arbitrarily large.
\end{cor}

The second result is the following necessary condition for being quasi-alternating.
\begin{thm}
\label{thm3}
If a knot $K$ is quasi-alternating, then $b^+_*(K) = b^-_*(K)=0$.
\end{thm}

In light of Theorem \ref{thm greene}, we say that a knot $K$ is {\it 4-dimensionally alternating} if
$b^+_*(K)=b^-_*(K)=0$.
Theorem \ref{thm3} says that if a knot $K$ is quasi-alternating, then $K$ is 4-dimensionally alternating.
Then, is the inverse also true? The answer is no; we classify 4-dimensionally alternating knots up to 10 crossings, and give 4-dimensionally alternating knots which are not concordant to any quasi-alternating knot.
\begin{prop}
\label{prop1}
The knots $10_{139}$, $10_{152}$, $10_{154}$ and $10_{161}$ in Rolfsen's table are 4-dimensionally alternating
but not concordant to any quasi-alternating knot.
For any other knot $K$ with up to 10 crossings,
$K$ is 4-dimensionally alternating if and only if $K$ is quasi-alternating or slice.
\end{prop}

\begin{remark}
If $K^*$ is the mirror of $K$, then $b^-(K)= b^+(K^*)$ and $b^-_*(K)= b^+_*(K^*)$. Hence we only need to
study $b^+(K)$ and $b^+_*(K)$, but we often use $b^-(K)$ and $b^-_*(K)$ for convenience.
\end{remark}

\begin{acknowledge}
The author was supported by JSPS KAKENHI Grant Number 15J10597.
\end{acknowledge}


\section{Relationship to various genera of knots}
In this section, we study the relationship of $b^{\pm}(K)$ and $b^{\pm}_*(K)$
to genera of knots. We start from orientable genera of knots.
The {\it 3-genus} $g_3(K)$ ( the {\it 4-genus} $g_4(K)$) of a knot $K$
is the minimal number of the genus of any orientable surface in $S^3$ (resp.\ in $B^4$)
with boundary $K$. 
Then Gordon-Litherland's theorem \cite{gordon-litherland} gives the following inequalities;
\begin{prop}
\label{prop2}
For any knot $K$, we have
$$
b^{\pm}(K) \leq g_3(K) \pm \frac{\sigma(K)}{2},
$$
and
$$
b^{\pm}_*(K) \leq g_4(K) \pm \frac{\sigma(K)}{2}.
$$
\end{prop}

\begin{proof}
In \cite{gordon-litherland},
Gordon and Litherland prove that
for any orientable surface $F$ in $B^4$ with boundary $K$,
we have 
$$
\sigma(M_F)= \sigma(K),
$$
where $\sigma(M_F) = b_2^+(M_F)-b^-_2(M_F)$.
In addition, let $b_i$ denote the $i$-th Betti number, and then we can verify that
$$
2g(F)=b_1(F)=b_2(M_F)=b_2^+(M_F)+b^-_2(M_F).
$$
By these equalities, we have 
$$
b_2^{\pm}(M_F) = \frac{1}{2}(b_2(M_F) \pm \sigma(M_F)) = g(F) \pm \frac{\sigma(K)}{2}.
$$
Hence, if $F$ is in $S^3$ and has genus $g_3(K)$, then we have
$$
b^{\pm}(K) \leq 
b_2^{\pm}(M_F) = g_3(K) \pm \frac{\sigma(K)}{2}.
$$
Similarly, if $F$ is in $B^4$ and has genus $g_4(K)$, we have
$$
b^{\pm}_*(K) \leq 
b_2^{\pm}(M_F) = g_4(K) \pm \frac{\sigma(K)}{2}.
$$
\end{proof}
Next we consider non-orientable genera of knots.
The {\it non-orientable 3-genus} $\gamma_3(K)$ of a knot $K$
is the minimal number of the first Betti number of any non-orientable surface in $S^3$
with boundary $K$. 
Then we have the following;
\begin{prop}
\label{prop3}
For any knot $K$, we have
$$
b^+(K) + b^-(K) \leq \gamma_3(K),
$$
and
$$
b^+_*(K) + b^-_*(K) \leq \gamma_4(K).
$$
\end{prop}
\begin{proof}
For any surface $F$ in $B^4$ with boundary $K$,
we can verify that 
$$
b_1(F)=b_2(M_F) = b^+_2(M_F)+b^-_2(M_F).
$$
Hence if $F$ is a non-orientable surface in $S^3$ with $\partial F =K$ and $b_1(F)=\gamma_3(K)$,
then we have
$$
\gamma_3(K)=b_1(F)=b^+_2(M_F)+b^-_2(M_F) \geq b^+(K)+b^-(K).
$$  
Similarly, we can prove that 
$\gamma_4(K) \geq b^+_*(K)+b^-_*(K)$.
\end{proof}


\section{Dealternating number and alternation number}
We next consider 
the dealternating number and the alternation number.
We first recall the definition of these invariants.
A knot diagram is {\it $n$-almost alternating} if $n$ crossing changes in the diagram
turn the diagram into an alternating knot diagram. 
We say that
a knot $K$ has the {\it dealternating number} $n$  
if $K$ has 
an $n$-almost alternating diagram 
and no $k$-almost alternating diagram for any $k<n$.
We denote the dealternating number of $K$ by $\dalt(K)$.
The {\it alternation number} $\alt(K)$ of a knot $K$
is the minimal number of the Gordian distance between $K$ and any alternating knot.
Then we have the following inequalities.
\begin{prop}
\label{prop4}
$
b^+(K) + b^-(K) \leq \dalt(K).
$
\end{prop}

\begin{prop}
\label{prop5}
$
b_*^+(K) + b_*^-(K) \leq 2 \lceil \frac{\alt(K)}{2} \rceil.
$
\end{prop}

The aim of this section is to prove the above two propositions.
To prove Proposition \ref{prop4},
we use the Goeritz form for surfaces in $S^3$, which is introduced in \cite{gordon-litherland}. 
For a surface $F$ in $S^3$, let $G_F$ denote the Goeritz form for $F$,
$\sigma(G_F)$ the signature of $G_F$, and $e(F)$ the Euler number of $F$.
Then it is proved in \cite{gordon-litherland} that
$
\sigma(G_F)= \sigma(M_F) = \sigma(K) - \frac{e(F)}{2}.
$ 

\def\proofname{Proof of Proposition \ref{prop4}}

\begin{proof}
Suppose that a diagram $D$ for a knot $K$ is deformed into an alternating diagram
$D'$ for a knot $K'$ by $n$ crossing changes. 
Note that $D$ and $D'$ have the same projection.
We choose an orientation of $K$ and a checkerboard coloring of $D$ arbitrarily,
and choose those of $K'$ and $D'$ so that the orientation
and coloring on the projection induced by $K'$ and $D'$ are equal to ones induced by
$K$ and $D$. Let $B$ (and $W$) denote the spanning surface for $K$ in $S^3$
dedicated by the black regions (resp.\ white regions) on $D$. 
Similarly, we take the spanning surfaces $B'$ and $W'$ for $K'$ from the checkerboard coloring of $D'$ respectively. 
Here we note that since $D'$ is an alternating diagram, one of $G_{B'}$ and $G_{W'}$ is 
positive definite and the other is negative definite.
We may assume that $G_{B'}$ is positive definite.

We consider the value of $\sigma(G_B)$ and $\sigma(G_W)$. 
On the diagrams $D$ and $D'$, we divide $n$ crossings performed crossing change
into two types; Type I and Type II in Figure \ref{type}.
In addition, we assign $+1$ or $-1$ to each crossing as shown in Figure \ref{sign},
which is called the {\it sign} of a crossing.  
Let $n_1$ (and $n_2$) denote the number of Type I (resp.\ Type II) crossings.
Then $n=n_1+n_2$. Moreover,  the Euler number of $B$ and $B'$ are computed by counting the sign of Type II crossings, and we see that 
$$
|\frac{e(B')}{2}-\frac{e(B)}{2}| \leq 2n_1.
$$
Similarly, the Euler number of $W$ and $W'$ are computed by counting the sign of Type I crossings and we have
$$
|\frac{e(W')}{2}-\frac{e(W)}{2}| \leq 2n_2.
$$
Now, since $b_1(B) = b_1 (B') = \sigma(G_{B'})$, we see that
\begin{eqnarray*}
\sigma(G_B)&=&\sigma(K) - \frac{e(B)}{2}\\
\ &\geq& \sigma(K) - \frac{e(B')}{2} -2n_1\\
\ &=& \sigma(K) - \sigma(K') + b_1(B) - 2n_1.
\end{eqnarray*}
This implies that
$$
n_1 \geq \frac{1}{2}(b_1(B)- \sigma(G_B))  + \frac{1}{2}(\sigma(K) - \sigma(K'))
= b^-_2(M_B) +\frac{1}{2}(\sigma(K) - \sigma(K')).
$$
Similarly, since $b_1(W) = b_1 (W') = -\sigma(G_{W'})$, we see that
$$
\sigma(G_W) \leq \sigma(K) - \sigma(K') - b_1(W) + 2n_2.
$$
This implies that
$$
n_2 \geq \frac{1}{2}(b_1(W) + \sigma(G_W))  + \frac{1}{2}(\sigma(K') - \sigma(K))
= b^+_2(M_W) +\frac{1}{2}(\sigma(K') - \sigma(K)).
$$
Since $n=n_1 + n_2$, we have
$$
n=n_1+n_2 \geq b^+_2(M_W) + b^-_2(M_B) \geq b^+(K) + b^-(K).
$$
This completes the proof.
\end{proof}

\begin{figure}[htbp]
\begin{minipage}[]{0.4\hsize}
\begin{center}
\includegraphics[scale = 0.7]{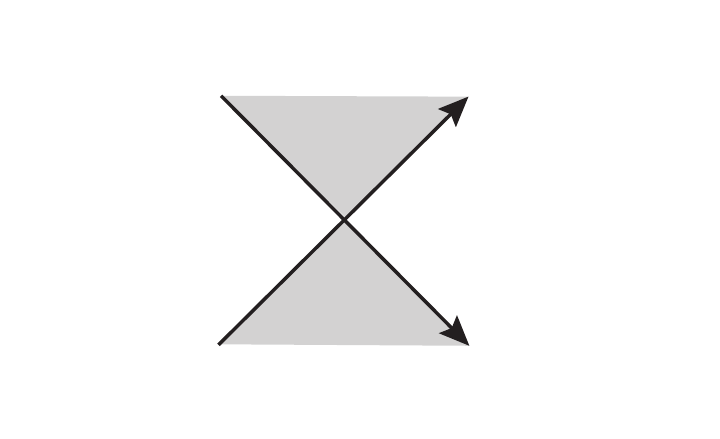}
Type I
\end{center}
 \end{minipage}
 \begin{minipage}[]{0.4\hsize}
\begin{center}
\includegraphics[scale = 0.7]{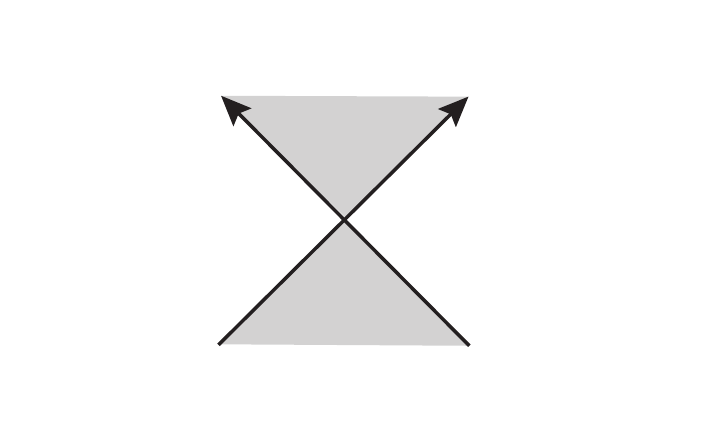}
Type II
\end{center}
 \end{minipage}
\caption{\label{type} Type of crossings}
\end{figure}

\begin{figure}[htbp]
\begin{minipage}[]{0.4\hsize}
\begin{center}
\includegraphics[scale = 0.7]{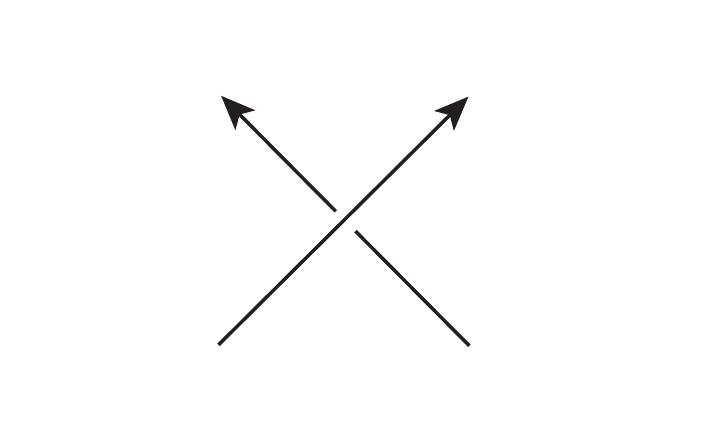}
$+1$
\end{center}
 \end{minipage}
 \begin{minipage}[]{0.4\hsize}
\begin{center}
\includegraphics[scale = 0.7]{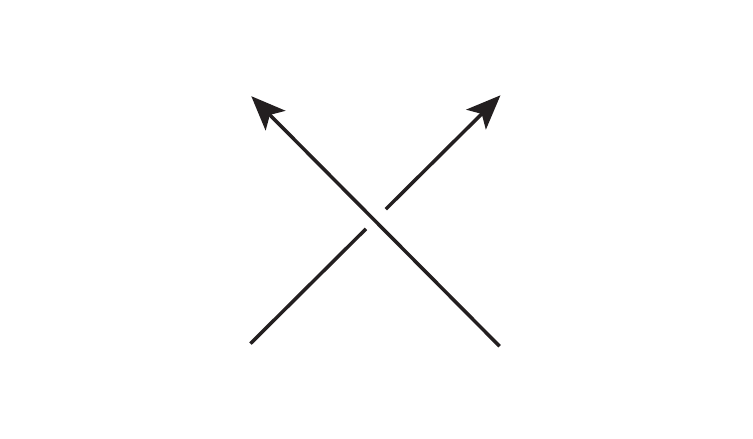}
$-1$
\end{center}
 \end{minipage}
\caption{\label{sign} Sign of crossings}
\end{figure}

Next we prove Proposition \ref{prop5}.

\def\proofname{Proof of Proposition \ref{prop5}}

\begin{proof}
Suppose that a knot $K$ is deformed into an alternating knot $J$ after $n$ crossing changes.
Then the regular homotopy dedicated by the crossing changes gives a self-transverse immersed annulus $A$ in $S^3$ $\times [0,1]$ such that 
the number of its self-intersections is $n$ and 
$(\partial (S^3 \times [0,1]), \partial A)$ is diffeomorphic to 
the disjoint union $(S^3,K) \amalg (S^3, J^*)$.
By the argument in the proof of \cite[Proposition 2.3]{murakami-yasuhara},
we can construct a (possibly non-orientable) cobordism $C$ in $S^3 \times [0,1]$
from $J$ to $K$ which satisfies
$b_1(C) = 2 \lceil \frac{n}{2} \rceil +1$.
Here, since $J$ is alternating, $J$ bounds two surfaces $P$ and $N$ in $B^4$ such that
$M_P$ and $M_N$ are positive definite and negative definite respectively.
Now, by gluing $C$ with $P$ and $N$ along $J$, we obtain two surfaces $C \cup P$
and $C \cup N$ in $B^4$ with boundary $K$.
Since $M_{C\cup P} = M_C \cup M_P$ and
$M_{C\cup N} = M_C \cup M_N$, it follows from elementary homology theory that 
\begin{enumerate}
\item $b^-_2(M_{C \cup P}) = b^-_2(M_{C})$,
\item $b^+_2(M_{C \cup N}) = b^+_2(M_{C})$, and
\item $b^+_2(M_C)+b^-_2(M_C) = b_1(C)-1$.
\end{enumerate}
These imply that 
$2\lceil \frac{n}{2} \rceil = b^+_2(M_{C\cup N})+b^-_2(M_{C\cup P}) \geq b^+_*(K)+b^-_*(K)$.
\end{proof}


\section{Proof of main theorems}

In this section, we prove Theorem \ref{thm2} and Theorem \ref{thm3}.
We start from Theorem \ref{thm2}.

\def\proofname{Proof of Theorem \ref{thm2}}

\begin{proof}
The inequality $\gamma_4(K) \geq b^+_*(K)$ is given by Proposition \ref{prop3}.
We prove $b^+_*(K) \geq \sigma(K)/2 - d(S^3_{-1}(K))$.

Let $F$ be a surface in $B^4$ with boundary $K$.
We first assume that $b_1(F)$ is odd.
Then $F$ is non-orientable, and 
it follows from \cite[Theorem 1.5]{batson} that 
\begin{equation}
b_1(F) \geq \frac{e(F)}{2} - 2d(S^3_{-1}(K)).
\end{equation}
Moreover, as mentioned in Section 3, the equality 
$$  
\sigma(M_F) = \sigma(K) - \frac{e(F)}{2}
$$
also holds.
Combining them, we have
\begin{equation}
\sigma(K)- 2d(S^3_{-1}(K)) \leq b_1(F) + \sigma (M_F) = 2b^+_2(M_F).
\end{equation}
Next we assume that $b_1(F)$ is even.
Then, by taking the boundary connected sum of $F$ with a M\"{o}bius band in $B^4$ with boundary the unknot and Euler number $+2$, we have a non-orientable surface $F'$ in $B^4$ with boundary $K$ such that $b_1(F')= b_1(F) +1$ and $e(F')= e(F)+2$.
By applying \cite[Theorem 1.5]{batson} to $F'$, we have
$$
b_1(F)+1 \geq \frac{e(F)+2}{2} - 2d(S^3_{-1}(K)).
$$ 
This inequality is equivalent to the inequality (1), and
hence the inequality (2) also holds in this case.
This completes the proof.
\end{proof}

Next, we prove Theorem \ref{thm3}.
We first recall quasi-alternating links.
For a link $L$, fix a diagram of $L$ and choose a crossing on the diagram.
Then we obtain two links $L_0$ and $L_1$ by replacing the crossing by
the two simplifications shown in Figure \ref{resolutions}. 
We call the links $L_0, L_1$ a {\it pair of resolutions} for $L$.
The set $\mathcal{Q}$ of {\it quasi-alternating links} is the smallest set of links which
satisfies the following properties:
\begin{enumerate}
\item the unknot is in $\mathcal{Q}$
\item the set $\mathcal{Q}$ is closed under the following operation. 
Suppose $L$ is any link which
has a pair of resolutions $L_0, L_1$ with the following properties:
\begin{itemize}
\item $L_0, L_1 \in \mathcal{Q}$,
\item$\det(L_0), \det(L_1) \neq 0$,
\item $\det(L) = \det(L_0) + \det(L_1)$;
\end{itemize}
then $L \in \mathcal{Q}$.
\end{enumerate}
Here $\det(L)$ denotes the determinant of $L$;
namely, if we denote the double branched cover of $S^3$ over $L$ by $\Sigma(L)$,
then 
$$
\det(L) := 
\left\{
\begin{array}{ll}
|H_1(\Sigma(L);\Bbb{Z})|&(\text{if $H_1(\Sigma(L);\Bbb{Z})$ is finite})\\
0&(\text{otherwise})
\end{array}
\right..
$$
\begin{figure}[htbp]
\hspace{-15mm}
\begin{minipage}[]{0.3\hsize}
\includegraphics[scale = 0.7]{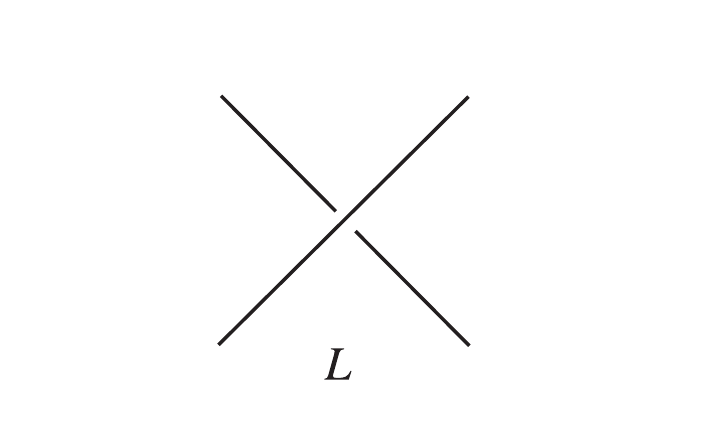}
\end{minipage}
\begin{minipage}[]{0.3\hsize}
\includegraphics[scale = 0.7]{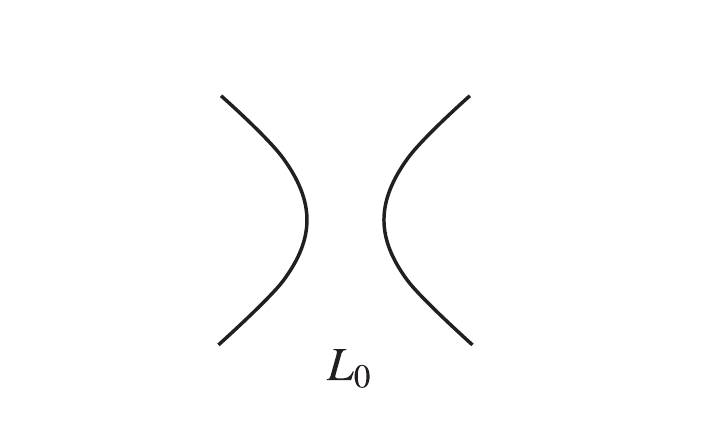}
 \end{minipage}
 \begin{minipage}[]{0.3\hsize}
\includegraphics[scale = 0.7]{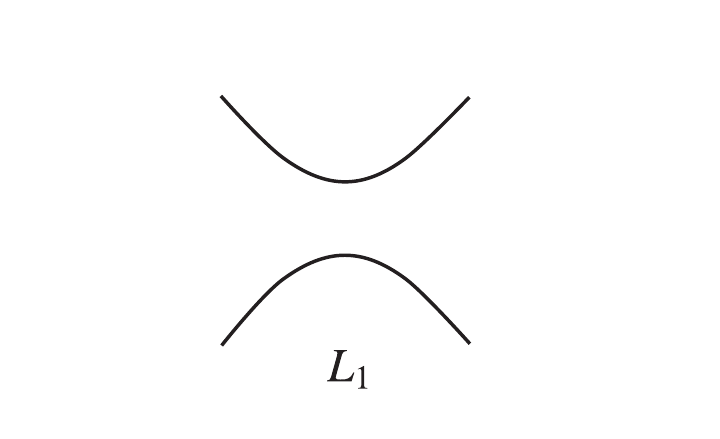}
 \end{minipage}
\caption{\label{resolutions} Resolutions for a link $L$}
\end{figure}
\begin{figure}[htbp]
\hspace{-15mm}
\begin{minipage}[]{0.3\hsize}
\includegraphics[scale = 0.7]{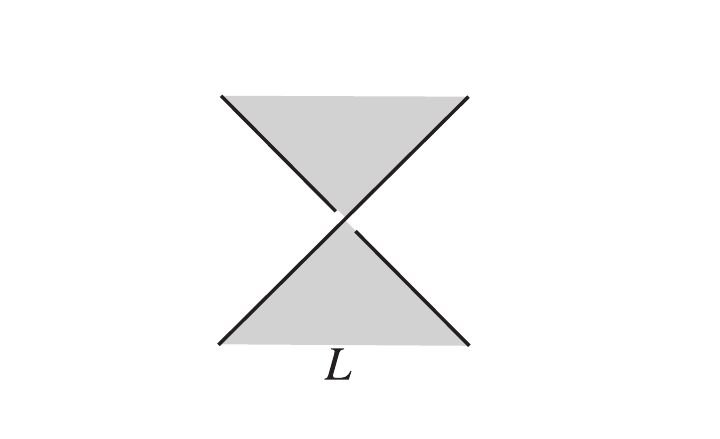}
\end{minipage}
\begin{minipage}[]{0.3\hsize}
\includegraphics[scale = 0.7]{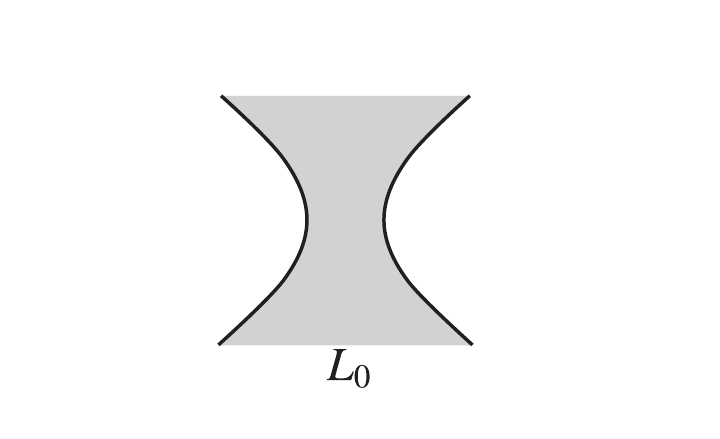}
 \end{minipage}
 \begin{minipage}[]{0.3\hsize}
\includegraphics[scale = 0.7]{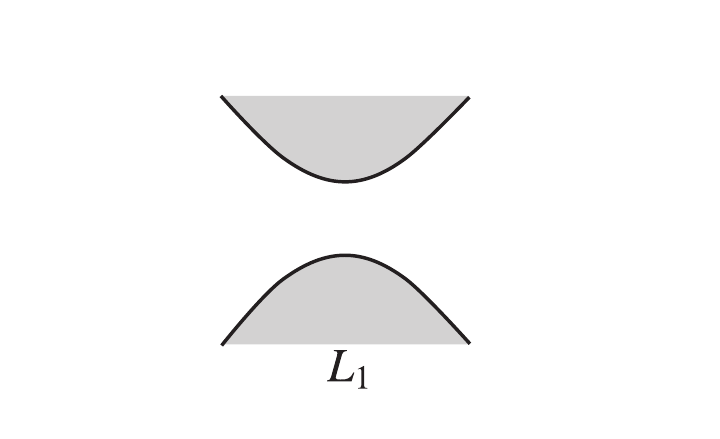}
 \end{minipage}
\caption{\label{coloring}}
\end{figure}

For a quasi-alternating link $L=L_{\emptyset}$ and its resolutions $L_0$, $L_1$ satisfying the property (2),
choose the checkerboard coloring of their diagrams 
as shown in Figure \ref{coloring}. 
For each $i \in \{ \emptyset, 0 ,1\}$,
let $B_{i}$ (and $W_{i}$) denote the spanning surface for $L_{i}$ in $S^3$
dedicate by the black regions (resp.\ white regions).
Then the differences $B_{\emptyset}-B_1$ and $W_{\emptyset}-W_0$ can be
regarded as cobordisms in $S^3 \times [0,1]$ from $L_0$ and $L_1$ to $L$
respectively.  We denote these cobordisms by $C_B$ and $C_W$ respectively.
Theorem \ref{thm3} follows from the following lemma.
\begin{lem}
\label{lem qa}
The double branched cover $M_{C_B}$ is positive definite, and $M_{C_W}$ is negative definite. 
\end{lem}

\def\proofname{Proof}
\begin{proof}
We first consider $M_{C_B}$.
For a given 4-manifold $M$, let $Q_M$ denote the intersection form of $M$.
Since $M_{B_{\emptyset}} = M_{B_1} \cup M_{C_B}$ and $\Sigma(L_1)$ is a rational homology 3-sphere,
the bilinear form $Q_{M_{B_{\emptyset}}}$ is isomorphic to 
$Q_{M_{B_1}} \oplus Q_{M_{C_B}}$ over $\Bbb{Q}$. 
Since $b_2(M_{C_B})= b_2(M_{B_{\emptyset}}) - b_2(M_{B_1}) =1$,
the 4-manifold $M_{C_B}$ is positive definite if and only if 
$\sigma(M_{C_B})=\sigma(M_{B_{\emptyset}}) - \sigma(M_{B_1}) =1$.

It is proved in \cite{gordon-litherland} that for any surface $F$ in $S^3$, 
its Goeritz form $G_F$ is isomorphic to $Q_{M_F}$, and so
we can prove $\sigma(M_{B_{\emptyset}}) - \sigma(M_{B_1}) =1$
by studying the Goeritz forms $G_{B_{\emptyset}}$ and  $G_{B_1}$.
Let $g_{B_1}$ be a representation matrix for $G_{B_1}$.
Note that $B_1$ is lying both in $B_{\emptyset}$ and in $B_0$, and there exist representation matrices
$g_{B_\emptyset}$ and $g_{B_0}$ for $G_{B_{\emptyset}}$ and $G_{B_0}$ such that
$$
g_{B_{\emptyset}}=
\left(
\begin{array}{cc}
a & \mathbf{b}\\
\mathbf{b}^{\tau} & g_{B_1}
\end{array}
\right)
\text{ and }
g_{B_0}=
\left(
\begin{array}{cc}
a-1 & \mathbf{b}\\
\mathbf{b}^{\tau} & g_{B_1}
\end{array}
\right)
$$
for some integer $a$ and row vector $\mathbf{b}$. 
Moreover, there exist $\Bbb{Q}$-coefficient square matrices $p$ and $q$
which satisfy
\begin{enumerate}
\item $\det p = \det q = 1$,
\item the product $p (g_{B_1}) p^{\tau}$ is a diagonal matrix, and
\item
$$
q(g_{B_{\emptyset}})q^{\tau}=
\left(
\begin{array}{cc}
a' & \mathbf{0}\\
\mathbf{0} & p(g_{B_1})p^{\tau}
\end{array}
\right)
\text{ and }
q(g_{B_0})q^{\tau}=
\left(
\begin{array}{cc}
a'-1 & \mathbf{0}\\
\mathbf{0} & p(g_{B_1})p^{\tau}
\end{array}
\right)
$$
for some rational number $a'$.
\end{enumerate}
This implies that $\sigma(M_{B_{\emptyset}}) - \sigma(M_{B_1}) =1$
if and only if $a'>0$.
We compare the determinant of $g_{B_{\emptyset}}$, $g_{B_0}$ and $g_{B_1}$
to prove $a'>0$.
It is known that for a 4-manifold $M$,
the determinant of any representation matrix for $Q_M$
is equal to the order of $H_1(\partial M ; \Bbb{Z})$.
Hence we see that
$$
\det(L_{i})= |\det(g_{B_i})|
$$
and
$$
\det(g_{B_{\emptyset}})
= a' \cdot \det(g_{B_1})
=\det(g_{B_0}) + \det(g_{B_1}).
$$
Since $\det(L)=\det(L_0)+ \det(L_1)$, 
the above equalities imply that $\det(g_{B_0})$ and $\det(g_{B_1})$ have the same sign,
and so $a' = \frac{\det(g_{B_0})}{\det(g_{B_1})}+1$ is positive.
Similarly, we can prove that $M_{C_W}$ is negative definite.
\end{proof}

\def\proofname{Proof of Theorem \ref{thm3}}

\begin{proof}
Here we prove Theorem \ref{thm3} for all quasi-alternating links;
name-ly, we prove that any quasi-alternating link bounds surfaces $P$ and $N$ in $B^4$
whose double branched cover $M_P$ and $M_N$ are positive definite and negative definite respectively. We prove this assertion by induction on $\det(L)$.

If $L$ is a quasi-alternating link with $\det(L)=1$, then $L$ is the unknot and obviously
4-dimensionally alternating. 
Suppose that $\det(L)=n>1$ and Theorem \ref{thm3} holds for any quasi-alternating 
link with determinant less than $n$.
Then there exists a pair of resolutions $L_0, L_1$ for $L$,
such that $L_0, L_1$ are quasi-alternating and $\det(L)= \det(L_0) + \det(L_1)$.
Since $\det(L_0), \det(L_1)<n$,
we can take spanning surfaces $N_0$ and $P_1$ for $L_0$ and $L_1$
in $B^4$ whose double branched cover are negative definite and positive definite respectively.
Therefore, it follows from Lemma \ref{lem qa} that $C_W \cup N_0$ (and $C_B \cup P_1$) is a spanning surface for $L$ in $B^4$
whose double branched cover is negative definite (resp.\ positive definite). 
This completes the proof.
\end{proof}


\section{Proof of Proposition \ref{prop1}}

In this section, we prove Proposition \ref{prop1}.
Note that any slice knot is obviously 4-dimensionally alternating, and
it follows from Theorem \ref{thm3} that any quasi-alternating knot is 4-dimensionally alternating.
Therefore, Proposition \ref{prop1} follows from the following proposition.
\begin{prop}
\label{prop4.1}
The knots $10_{139}$, $10_{152}$, $10_{154}$ and $10_{161}$ are 
4-dimensionally alternating but not concordant to any quasi-alternating knot.
For any knot $K$ with 10 or fewer crossings except for the above four knots,
if $K$ is neither quasi-alternating nor slice,
then $K$ is not 4-dimensionally alternating.
\end{prop}
Moreover, it is described in \cite{champanerkar-kofman}
that a knot $K$ with 10 or fewer crossings is not quasi-alternating
if and only if $K$ is one of the following 14 knots
$$
\begin{array}{c}
8_{19}, 9_{42}, 9_{46}, 10_{124}, 10_{128}, 10_{132}, 10_{136},\\ 
10_{139}, 10_{140}, 10_{145}, 10_{152}, 10_{153}, 10_{154}, \text{ and } 10_{161},
\end{array}
$$
where $9_{46}$, $10_{153}$ and $10_{140}$ are slice knots.
Taking these facts into consideration, we decompose Proposition \ref{prop4.1}
into the following three lemmas.
\begin{lem}
\label{lem6}
If $K$ is one of  $8^*_{19}$, $9^*_{42}$, $10^*_{128}$ and $10^*_{136}$,
then $\frac{\sigma(K)}{2}-d(S^3_{-1}(K))=1$,
and hence $K$ is not 4-dimensionally alternating.
\end{lem}
\begin{lem}
\label{lem7}
If $K$ is one of $10_{124}$, $10_{132}$ and $10_{145}$,
then $\Sigma(K)$ cannot bound any positive definite 4-manifold,
and hence $K$ is not 4-dimensionally alternating.
\end{lem}
\begin{lem}
\label{lem5}
The knots $10_{139}$, $10_{152}$, $10_{154}$ and $10_{161}$ are 4-dimensionally
alternating but not concordant to any quasi-alternating knot.
\end{lem}
We first prove Lemma \ref{lem6}.

\def\proofname{Proof of Lemma \ref{lem6}}
\begin{proof}
Lemma \ref{lem6} in the case of $8^*_{19}$ and $9^*_{42}$ are proved in \cite{batson}
and \cite{sato-tange} respectively.
We prove the lemma for $10^*_{128}$ and $10^*_{136}$.
It is easy to check that $\sigma(10^*_{128})=6$ and $\sigma(10^*_{136})=2$.
To compute $d(S^3_{-1}(K))$,
we use Ni-Wu's $V_k$-sequence \cite{ni-wu}.
Here $V_k$ is a $\Bbb{Z}_{\geq0}$-valued concordance invariant for each $k \in \Bbb{Z}_{\geq0}$. In particular, it follows from \cite[Proposition 1.6]{ni-wu} that
$$
d(S^3_{-1}(K))= -d(S^3_1(K^*)) = 2V_0(K^*).
$$
We compute $V_0$ of $10_{128}$ and $10_{136}$ by using the following proposition,
which immediately follows from \text{\cite[Proposition 1.9]{sato}}. 
Here we denote $\overline{\Bbb{C}P^2}$ with open 4-ball deleted by 
$\punc \overline{\Bbb{C}P^2}$.
\begin{prop}
\label{cp2}
Suppose that a knot $K$ bounds a disk $D$ in $\punc \overline{\mathbb{C}P^2}$
such that 
$[D,\partial D] = n \gamma \in H_2(\punc \overline{\mathbb{C}P^2}, \partial (\punc \overline{\mathbb{C}P^2}); \mathbb{Z})$ for a generator $\gamma$ and some odd 
integer $n>0$. Then we have
$$
V_0(K)= \frac{(n-1)(n+1)}{8}.
$$
\end{prop}
As shown in Figure \ref{disk128} and \ref{disk136}, we see that 
$10_{128}$ and $10_{136}$ bounds a disk in $\punc \overline{\Bbb{C}P^2}$
with $n=3$ and $n=1$. By Proposition \ref{cp2}, we have
$V_0(10_{128})=1$ and $V_0(10_{136})=0$. This completes the proof.
\end{proof}

\begin{figure}[htbp]
\includegraphics{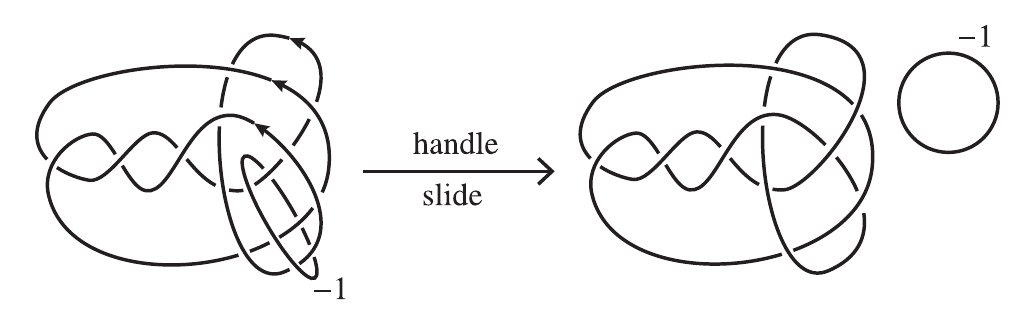}
\caption{\label{disk128} $10_{128}$ bounds a disk in $\punc \overline{\Bbb{C}P^2}$ with $n=3$}
\includegraphics{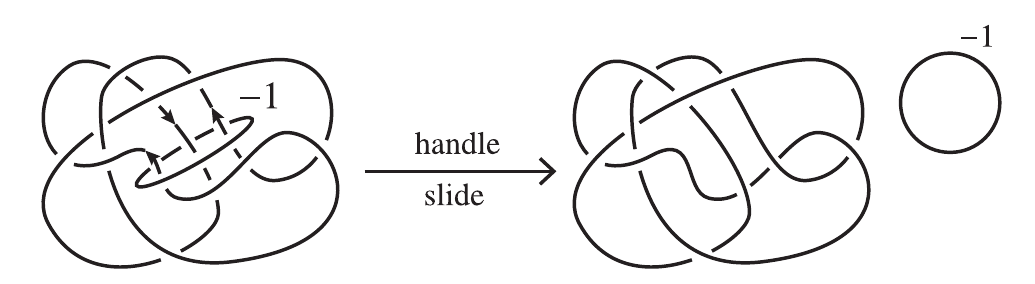}
\caption{\label{disk136} 
$10_{136}$ bounds a disk in $\punc \overline{\Bbb{C}P^2}$ with $n=1$}
\end{figure}

In order to prove Lemma \ref{lem7} and Lemma \ref{lem5},
we use Akbulut's method in \cite{akbulut} for describing a handle diagram for 
the double branched cover of $B^4$ over any ribbon surface.   

\def\proofname{Proof of Lemma \ref{lem7}}
\begin{proof}
We can verify that the boundary of the ribbon surfaces in Figure \ref{124},
Figure \ref{132} and Figure \ref{145} are
$10_{124}$, $10_{132}$ and $10_{145}$ respectively.
By applying Akbulut's method to these ribbon surfaces, we see that 
$\Sigma(10_{124})= S^3_{-1}(3_1^*)=-S^3_1(3_1)$,
$\Sigma(10_{132})= S^3_{-5/2}(3_1^*)=-S^3_{5/2}(3_1)$
and
$\Sigma(10_{145})= S^3_{-3}(5_2^*)=-S^3_{3}(5_2)$.
It is proved in \cite{owens-strle} that for  any $r \in \Bbb{Q}_{>0}$,
the manifold $S^3_{r}(3_1)$ bounds a negative definite definite 4-manifold
if and only if  $r \geq 4$. Hence neither $\Sigma(10_{124})$ nor
$\Sigma(10_{132})$ can bound a positive definite 4-manifold.

Assume that $\Sigma(10_{145})=-S^3_{3}(5_2)$ bounds a positive definite 4-manifold $M$,
and let $C$ be the cobordism represented by the relative handle diagram in Figure \ref{5_2}.
Since $C$ is negative definite and $\partial C = -S^3_{3}(5_2) \amalg S^3_{3}(3_1)$,
the manifold $-M \cup C$ is a negative definite and has boundary $S^3_{3}(3_1)$. 
This contradicts the above result of \cite{owens-strle}, and hence $\Sigma(10_{145})$ cannot bound
any positive definite 4-manifold. 
\end{proof}

\begin{figure}[htbp]
\hspace{7mm}
\includegraphics{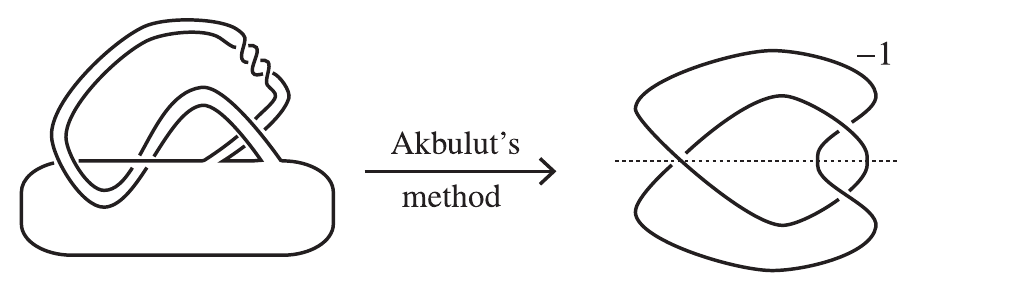}
\caption{\label{124} a ribbon surface with boundary $10_{124}$}
\hspace{7mm}
\includegraphics{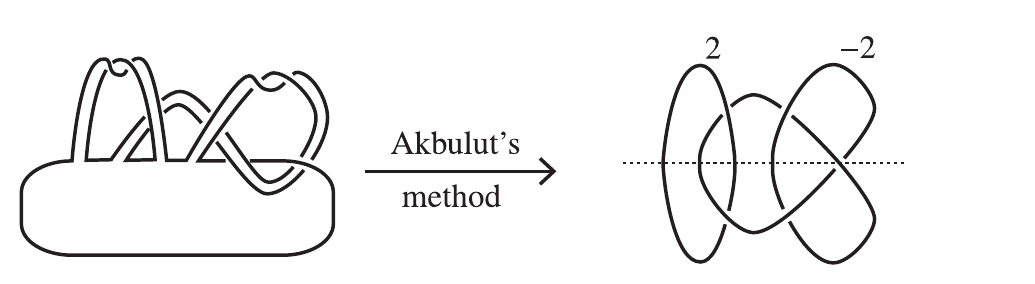}
\caption{\label{132} a ribbon surface with boundary $10_{132}$}
\hspace{7mm}
\includegraphics{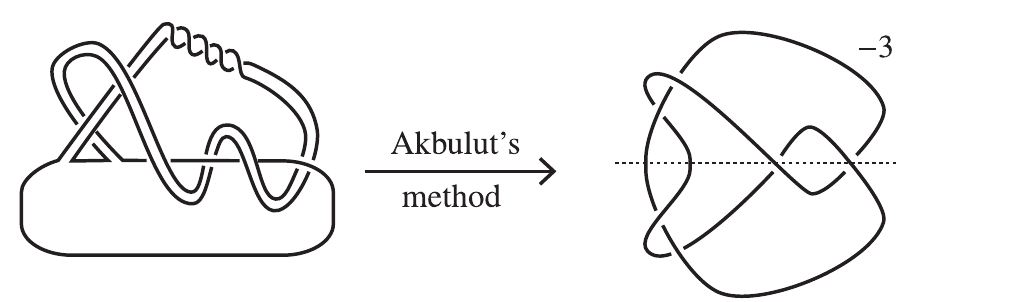}
\caption{\label{145} a ribbon surface with boundary $10_{145}$}
\includegraphics{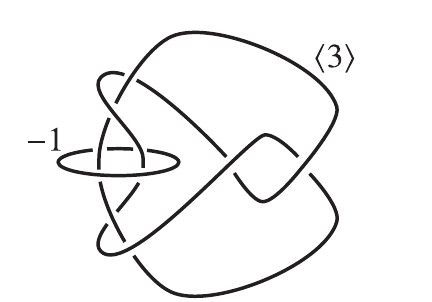}
\caption{\label{5_2} a negative definite cobordism from $S^3_{3}(5_1)$ to $S^3_{3}(3_1)$}
\end{figure}

Finally, we prove Lemma \ref{lem5}.
To prove Lemma \ref{lem5}, we use {\it $H(2)$-move},
which is a deformation of link diagram shown in Figure \ref{H(2)}.
If a diagram $D_1$ for a knot $K_1$ is deformed into a diagram $D_2$ for a knot $K_2$ by an $H(2)$-move,
then we can naturally obtain a non-orientable cobordism $C$ from $K_1$ to $K_2$ in $S^3 \times [0,1]$
such that $b_2(M_C)=1$. Moreover, 
if we denote the writhe of a knot diagram $D$ by $w(D)$, 
then the equality $e(C)= - (w(D_2)-w(D_1))$ holds, and hence 
Gordon-Literland's formula induces the equality
$$
\sigma(M_C)= (\sigma(K_2) -\sigma(K_1)) + \frac{1}{2}(w(D_2) -w(D_1)).
$$
We say that an $H(2)$-move is {\it positive} ({\it negative})
if $\sigma(M_C)=1$ (resp.\ $\sigma(M_C)=-1$).

\begin{figure}[htbp]
\includegraphics{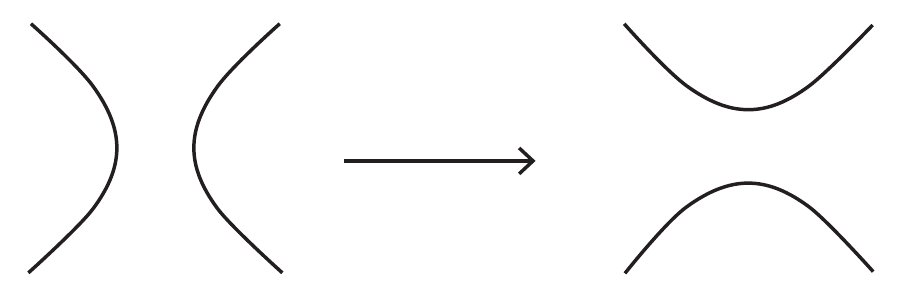}
\caption{\label{H(2)} $H(2)$-move}
\end{figure}

\def\proofname{Proof of Lemma \ref{lem5}}
\begin{proof}
Let $\tau$ be Ozsv\'ath-Szab\'o's $\tau$-invariant \cite{tau}.
It follows in \cite{knot info} that the knots $10_{139}$, $10_{152}$, $10_{154}$ and $10_{161}$
do not satisfy $\tau(K)=- \frac{\sigma(K)}{2}$, and hence they are not concordant to
any quasi-alternating knot.
We prove that these knots satisfy $b^{\pm}_*(K)=0$.
Here we consider the case of $10_{139}$.
It immediately follows from the lower diagrams in Figure \ref{139}
that $b^+_*(10_{139})=0$. Moreover, the knot $10_{139}$ is obtained from $3_1^*$
by the positive $H(2)$-move shown in the upper diagrams of Figure \ref{139}.
Since $3_1^*$ bounds a surface in $B^4$ whose double branched cover is positive definite,
we have $b^-_*(10_{139})=0$.  
Similarly, Figure \ref{152}, Figure \ref{154} and Figure \ref{161} shows
that $10_{152}$, $10_{154}$ and $10_{161}$ satisfy $b^{\pm}_*(K)=0$.
\end{proof}

\begin{figure}[htbp]
\includegraphics{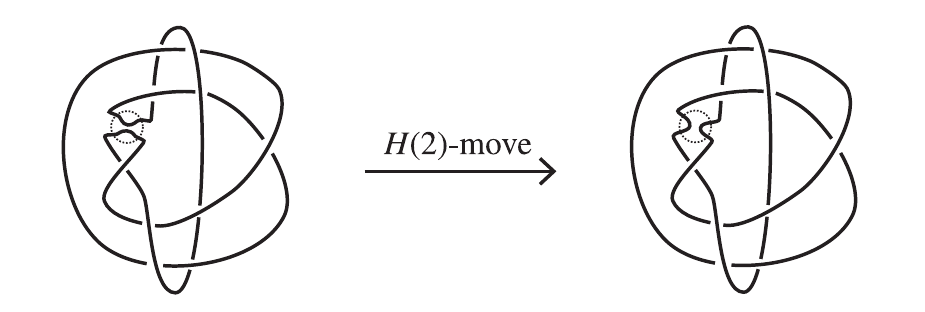}
\includegraphics{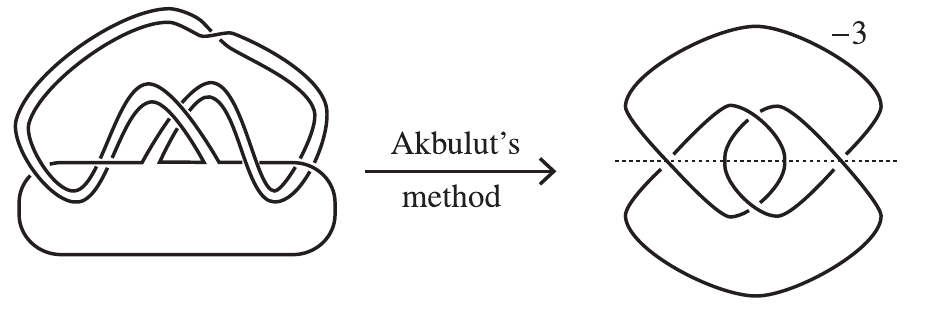}
\caption{\label{139} ribbon surfaces with boundary $10_{139}$}
\end{figure}

\begin{figure}[htbp]
\includegraphics{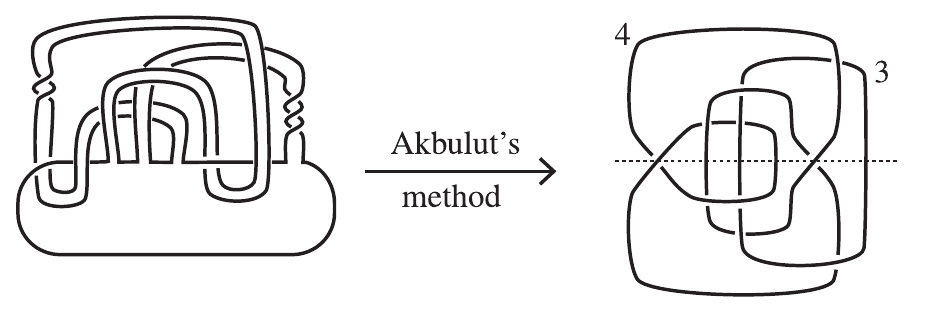}
\includegraphics{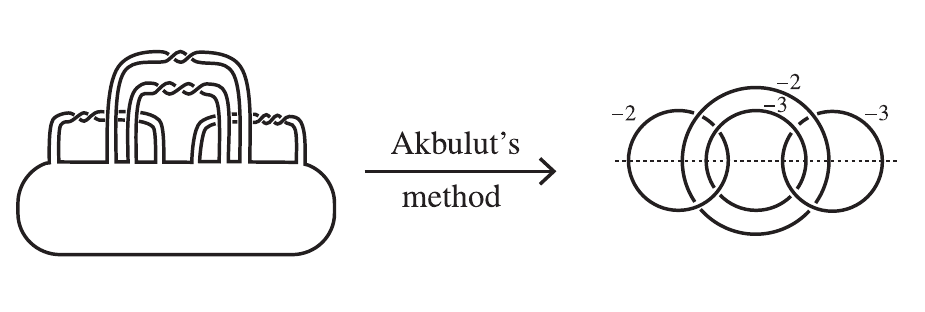}
\caption{\label{152} ribbon surfaces with boundary $10_{152}$}
\end{figure}

\begin{figure}[htbp]
\includegraphics{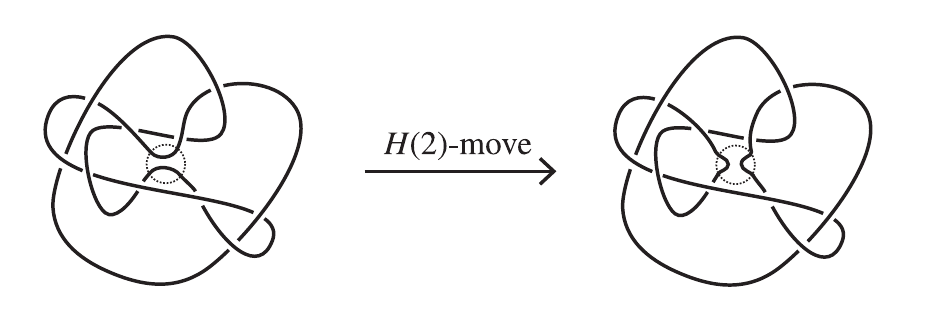}
\includegraphics{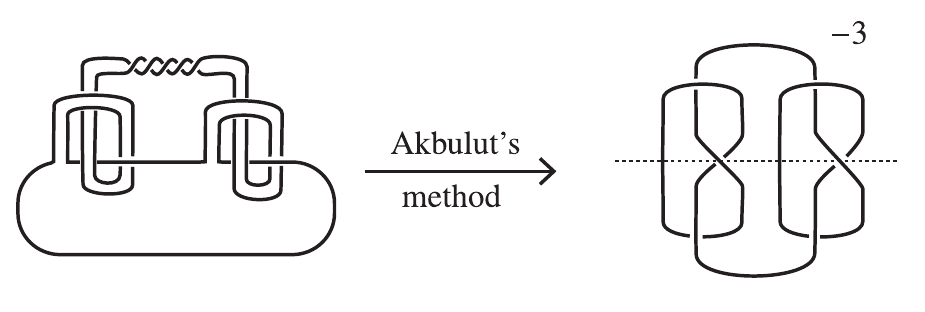}
\caption{\label{154} ribbon surfaces with boundary $10_{154}$}
\end{figure}

\begin{figure}[htbp]
\includegraphics{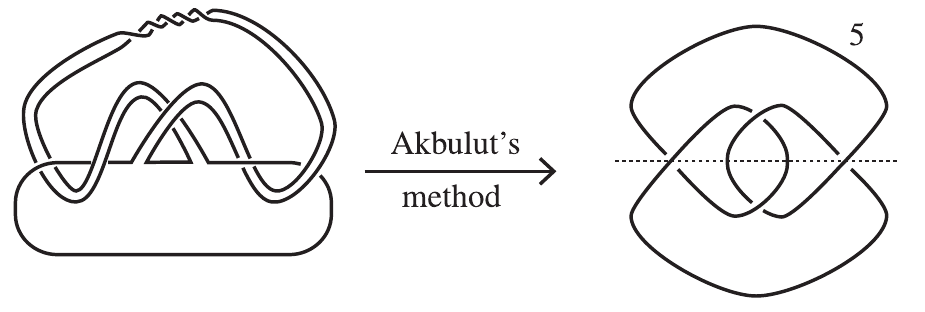}
\includegraphics{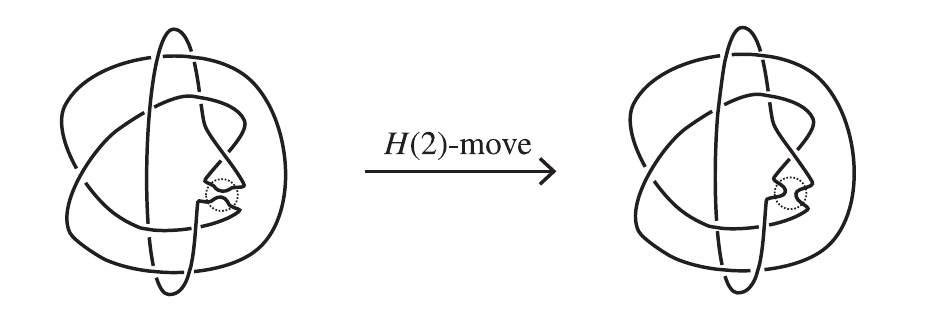}
\caption{\label{161} ribbon surfaces with boundary $10_{161}$}
\end{figure}

\end{document}